\documentclass[10pt]{amsart}  
\usepackage{amsmath}
\usepackage{amscd}
\usepackage{amssymb}
\usepackage{amsthm}
\RequirePackage{filecontents}
\usepackage{fullpage}
\usepackage[numbers]{natbib}  
\usepackage[draft]{hyperref}

\theoremstyle{plain}
\newtheorem{prop}{Proposition}
\newtheorem{thrm}[prop]{Theorem}
\newtheorem{lem}[prop]{Lemma}

\theoremstyle{definition}

\title{Short proof of Rademacher's formula for partitions}

\author{Wladimir de Azevedo Pribitkin}    
\address{Department of Mathematics \\ College of Staten Island, CUNY \\ Staten Island, NY 10314 \newline 
\indent \&  Department of Mathematics \\ The Graduate Center, CUNY \\ New York, NY 10016}  
\email{Wladimir.Pribitkin@csi.cuny.edu; w\_pribitkin@msn.com}

\author{Brandon Williams}
\address{Department of Mathematics \\ University of California \\ Berkeley, CA 94720}
\email{btw@math.berkeley.edu}


\begin{document}

\maketitle

\nocite{*}

\begin{abstract} This note rederives a formula for $r$-color partitions, $1 \le r \le 24$, including Rademacher's celebrated result for ordinary partitions, 
from the duality between modular forms of weights $-r/2$ and $2+r/2$.
\end{abstract}

\section{Introduction}

Let $\eta(\tau) = q^{1/24} \prod_{n=1}^{\infty} (1 - q^n)$, where $q = e^{2\pi i \tau}$ and $\tau \in \mathbb{H} = \{x+iy :  y > 0\}$, denote the 
\textbf{Dedekind eta-function}, a highly familiar cusp form of weight $1/2$. To be more precise, let $\Gamma = SL_2(\mathbb{Z}) $ be the full modular group, 
which is generated by $S = \big[{1 \atop 0}\,\,{1 \atop 1}\big]$ and $T = \big[{0 \atop 1}\,\,{-1 \atop 0}\big]$. Then $\eta$ transforms by
$$\eta(M \tau) = v_\eta(M) \sqrt{c \tau + d} \; \eta(\tau),$$
for all $M = \big[{* \atop c}\,\,{* \atop d}\big] \in \Gamma$, where the 
\textbf{eta multiplier} $v_\eta \colon \Gamma  \xrightarrow{\mbox{\tiny onto}} \{ \mbox{the $24$\textsuperscript{th} roots of unity} \}$ 
is determined by $v_\eta(S) = e^{\pi i / 12}$ and $v_\eta(T) = e^{-\pi i / 4}$. Here and below, $M$ acts upon $\mathbb{H}$ as usual and the roots are extracted
according to the convention that $ - \pi \le \mbox {arg } z < \pi$, for  $z \in \mathbb{C}^{\times}$. Note that $v_\eta^2$ is a character on $\Gamma$ and $v_\eta$
itself corresponds to a character on the metaplectic group $Mp_2(\mathbb{Z})$. Two different closed formulas are known for $v_\eta$ 
(see, for example, \cite[Section 6]{A} and \cite[Chapter 4]{K}).  \\      

Euler observed that the Fourier coefficients of $\eta^{-1}$ are very interesting: 
$$\eta(\tau)^{-1} = q^{-1/24} \sum_{n=0}^{\infty} p(n) q^n,$$ 
where the \textbf{partition function} $p(n)$ counts the number of ways to write $n > 0$ as an unordered sum of positive integers and $p(0) = 1$ by convention. 
More generally, 
$$\eta(\tau)^{-r} = q^{-r/24} \sum_{n=0}^{\infty} p_r(n) q^n,$$ 
where $p_r(n)$ enumerates the number of \textbf{$r$-color partitions} of $n > 0$, which permit parts to appear in $r$ different `colors' (with their order disregarded), and $p_r(0)=1$.
The modularity of $\eta$ is a powerful tool in the study of partitions. For example, Poisson summation shows that the series    
$q^{1/24} \sum_{n \in \mathbb{Z}} (-1)^n q^{n(3n - 1)/2}$ is a cusp form of the same weight and multiplier on $\Gamma$, so Euler's Pentagonal Number Theorem 
$$\eta(\tau) = q^{1/24} \sum_{n \in \mathbb{Z}} (-1)^n q^{n(3n - 1)/2}$$ 
follows after comparing only the coefficients of $q^{1/24}$ on both sides. In turn this identity implies Euler's recursive formula for $p(n)$. 
A much newer result of Bruinier and Ono \cite{BO} finds an intriguing finite algebraic formula for $p(n)$ in terms of traces of singular moduli of a distinguished weak Maass form of level $6$. 
For purposes of computation, however, the most remarkable result is the following one due to Rademacher. 
\begin{thrm} 
If $n \in \mathbb{N}$, then 
$$p(n) = \frac{1}{\pi \sqrt{2}} \sum_{c=1}^{\infty}  A_c(n) \sqrt{c} \, \frac{d}{dn} \bigg[ \frac{\sinh ( \mu \sqrt{n - 1/24} /c )}{\sqrt{n - 1/24}} \bigg],$$ 
where $\mu = \pi \sqrt{2/3}$ and $A_c(n) = e^{\pi i / 4} A(-1/24,n-1/24;c)$, with $A(\cdot,\cdot;c)$ defined in Lemma 2.  
\end{thrm}    

This infinite series converges rapidly and is the basis of modern algorithms for calculating $p(n)$ (see, for instance, \cite[Section 56.13]{HJP}). 
Using a closed formula for $v_\eta(M)$, one can express  
the finite sum $A_c(n)$ in terms of Dedekind sums (see, for example, \cite[Equations~3, 4]{P1}).
There is also a simpler (and more illuminating) formula for $A_c(n)$ due to Selberg \cite[Equation~18]{AS},  
whose work is fleshed out in \cite[Lectures~22--23]{R3}.  \\   

Rademacher's original derivation \cite{R1} and subsequent refinement \cite{R2} (also in \cite[Lectures~16--19]{R3}), which involves integrating over arcs of Ford circles,
are based upon the iconic circle method of Hardy and Ramanujan \cite{HR}. A more recent approach uses the observation of Hejhal \cite[Appendix D]{H}, 
rooted in the work of Niebur \cite{N}, that negative-weight modular forms can be studied using real-analytic Poincar\'e series (see also \cite[Section 6.3]{BFOR}). 
There are other ``real-analytic proofs": \cite{P1} writes $\eta^{-1}$ as a special value of a pseudo-Poincar\'e series of weight $-1/2$, while \cite{AA} constructs a weight $5/2$ mock modular
form with $\eta^{-1}$ as its shadow. \\

In this note we give a short ``holomorphic proof" of Rademacher's formula for $p(n)$ (as well as for $p_r(n)$, $2 \le r \le 24$) 
that uses only the Fourier expansion of Poincar\'e series and the fact that any weight $2$
modular form has constant term $0$. This derivation was carried out by the first author in the early 2000s and by the second one relatively recently. 
The former was inspired by the work of Siegel \cite{S}; the latter was guided by that of Zagier \cite{Z}, who credits Kaneko with an easier proof. 
Our paper presents a collaborative exposition.

\section{Review of Poincar\'e series and weight two modular forms}

The Fourier expansion of Petersson's Poincar\'e series \cite[Equations~10, 11]{HP} extends the pioneering work of Poincar\'e \cite[Section 6]{P}. We begin this section by recalling the 
computation pertinent to our study.  \\  

Let $\Gamma_{\infty}$ be the stabilizer of $i \infty$ in $\Gamma$, that is, the subgroup generated by $S$ and $T^2 = -I$.
For any function $f$ on $\mathbb{H}$, define the Petersson slash operator by $(f |_{k,v} M)(\tau) = v(M)^{-1} (c \tau + d)^{-k} f(M \tau)$,
where $M = \big[{* \atop c}\,\,{* \atop d}\big]  \in \Gamma$ and $v$ is a multiplier system in weight $k$ for $\Gamma$. Now fix $k \in \frac{1}{2}\mathbb{Z}$, $k \ge 5/2$,
and suppose that $m \in \frac{1}{24}\mathbb{Z}$ satisfies $k - 12m \in 2 \mathbb{Z}$. Then $q^m = e^{2\pi i m \tau}$ is invariant under $|_{k,v_\eta^{24m}} S$ and $|_{k,v_\eta^{24m}} (-I)$; 
hence the \textbf{Poincar\'e series} of weight $k$ and index $m$, 
$$P_{k,m}(\tau) = \sum_M q^m \Big|_{k,v_\eta^{24m}} M,$$
where $M$ ranges over a set of coset representatives for $\Gamma_{\infty} \backslash \Gamma$, is well-defined. This series converges absolutely-uniformly on compacta of 
$\mathbb{H}$, and $P_{k,m}$ is a modular form of weight $k$ and multiplier system $v_\eta^{24m}$ on $\Gamma$. It transforms under $S$ by $P_{k,m}(\tau + 1) = e^{2\pi i m} P_{k,m}(\tau)$
and thus has a ``Fourier expansion'' of the type 
$$P_{k,m}(\tau) = \sum_{n \in \mathbb{Z} + m} c_n q^n, \,\, {\mbox{\rm{where}}} \,\, c_n =  \int_0^1 P_{k,m}(\tau) e^{-2\pi i n \tau} \, dx,$$
$\tau = x+iy$, with $y>0$ fixed. Let us compute these Fourier coefficients.  
Using the bijection between $\Gamma_{\infty} \backslash \Gamma$ and the set $\{(0,1)\} \cup \{ (c,d) \in \mathbb{N} \times \mathbb{Z} :  (c,d) = 1 \}$
to choose the coset representatives $M = \big[{a \atop c}\,\,{* \atop d}\big] $, interchanging the sum and integral, and then splitting the summation over $d$ into residue classes, we see that
\begin{align*} 
c_n &= \delta_{m,n} + \sum_{c=1}^{\infty} \sum_{\substack{d \in \mathbb{Z} \\ (c,d) = 1}} v_\eta(M)^{-24m} \int_0^1 (c \tau + d)^{-k} e^{2\pi i (m M\tau - n \tau)} \, dx \\ 
&= \delta_{m,n} + \sum_{c=1}^{\infty} \sum_{\substack{h = 0 \\ (c,h) = 1}}^{c-1} v_\eta(M_{c,h})^{-24m}  \sum_{j \in  \mathbb{Z}} \int_{j}^{j + 1} (c \tau + h)^{-k} e^{2\pi i (m M_{c,h}\tau  - n\tau)} \, dx \\ 
&=  \delta_{m,n} + \sum_{c=1}^{\infty} \sum_{\substack{h = 0 \\ (c,h) = 1}}^{c-1} v_\eta(M_{c,h})^{-24m} e^{2\pi i \frac{ma + nh}{c}} 
	\int_{-\infty}^{\infty} (c \tau + h)^{-k} e^{2\pi i [m (M_{c,h}\tau - \frac{a}{c}) - n(\tau + \frac{h}{c})]} \, dx \\ 
&= \delta_{m,n} + \sum_{c=1}^{\infty} c^{-k} A(m, n; c) \int_{-\infty + iy}^{\infty + iy} \tau^{-k} e^{-2\pi i (n \tau + m / c^2 \tau)} \, d\tau, 
\end{align*} 
\noindent where the finite sum $A(m,n;c)$ is defined in Lemma 2.
Here $\delta_{m,n}$ is the Kronecker delta function and $M_{c,h} = \big[{a \atop c}\,\, {* \atop h}\big] \in \Gamma$, where we let $d=h+cj$, with $0 \le h \le c-1$, $(c,h)=1$, and $j \in \mathbb{Z}$.
Note that $M=M_{c,h} \cdot \big[{1 \atop 0}\,\, {j \atop 1}\big]$ and so $v_\eta(M)^{-24m} = e^{-2\pi i m j} \, v_\eta(M_{c,h})^{-24m}$.    
Along the way we replaced $x$ by $x - j$, invoked the identity $M_{c,h}\tau = \frac{a}{c} - \frac{1}{c(c\tau+h)}$,  
and then replaced $\tau$ by $\tau - \frac{h}{c}$. It remains for us to scrutinize the integral, say $\mathcal{I}$, which is independent of $y > 0$.  \\

If $n \le 0$, then we push the horizontal path of integration upwards to reveal that $\mathcal{I} = 0$.
If $n > 0$, then we let $w = -2\pi in \tau$, replace $y$ by $\frac{y}{2\pi n}$, expand the (appropriate) exponential as a power series, and integrate term by term to find that 
\begin{align*}
\mathcal{I} &= (2\pi n)^{k-1} i^{-k-1}\int_{y-i\infty}^{y+i\infty} w^{-k} e^{w-4\pi^2mn/c^2w} \, dw \\ 
&= (2\pi n)^{k-1} i^{-k-1} \sum_{p=0}^{\infty} \frac{(-4\pi^2mn/c^2)^p}{p!}\int_{y-i\infty}^{y+i\infty} \frac{e^w}{w^{p+k}} \, dw \\ 
&= 2\pi i^{-k}(2\pi n)^{k-1} \sum_{p=0}^{\infty} \frac{(-4\pi^2mn/c^2)^p}{p! \, \Gamma(p+k)} \\
&= \left \{ 
		 \begin{array}{ll}2\pi i^{-k} (c^2 n/m)^{\frac{k-1}{2}}J_{k-1}(4\pi \sqrt{m n}/c)
		 		   &  \, \mbox{if } \,  m \ne 0, \vspace{.125in} \\
				       (-1)^{k/2} (2\pi)^k n^{k-1}/(k-1)! 
			           &  \, \mbox{if } \,  m  = 0, \end{array} \right.  
\end{align*}  
\noindent where we resurrected Laplace's integral formula for the reciprocal of the gamma function
(see \cite{P2} or compute the special value of the inverse Laplace transform {\it de novo}) and then used the infinite series representation
of the ordinary Bessel function $J_{k-1}$(for the case $m \ne 0$). Finally, since the modified Bessel function $I_{k-1}$ satisfies $I_{k-1}(z) = i^{1-k} J_{k-1}(iz)$,
we recover the following trichotomy of Fourier expansion formulas.
\begin{lem} Let $k \in \frac{1}{2}\mathbb{Z}$, $k \ge 5/2$,
and suppose that $m \in \frac{1}{24}\mathbb{Z}$ satisfies $k - 12m \in 2 \mathbb{Z}$. 

\noindent (i) If $m > 0$, then ${ \displaystyle P_{k,m}(\tau) = \sum_{n \in (\mathbb{Z} + m)_{>0}} c_n q^n }$, where
$$c_n = \delta_{m,n} + 2\pi i^{-k} (n/m)^{\frac{k-1}{2}} { \displaystyle \sum_{c=1}^{\infty} \frac{A(m,n;c)}{c} \, J_{k-1}(4\pi \sqrt{mn}/c) }.$$
(ii) If $m < 0$, then ${ \displaystyle P_{k,m}(\tau) = q^m + \sum_{n \in (\mathbb{Z} + m)_{>0}} c_n q^n }$, where
$$c_n =   2\pi i^{-k} (n/|m|)^{\frac{k-1}{2}} { \displaystyle \sum_{c=1}^{\infty} \frac{A(m,n;c)}{c} \, I_{k-1}(4\pi \sqrt{|m|n}/c) }.$$
(iii) If $m = 0$, then ${ \displaystyle P_{k,0}(\tau) = 1 + \sum_{n = 1}^{\infty} c_n q^n }$, where
$$c_n =  { \displaystyle \frac{ (-1)^{k/2} (2\pi)^k n^{k-1}}{(k-1)!} \sum_{c\,=\,1}^{\infty} \frac{A(0,n;c)}{c^k} 
						= \frac{(-1)^{k/2} (2\pi)^k}{(k-1)! \, \zeta(k)} \, \sigma_{k-1}(n) = -\frac{2k}{B_k} \, \sigma_{k-1}(n) }.$$
Here $A(m,n;c)$ denotes the \textbf{generalized Kloosterman sum} 
$$A(m,n;c) = \sum_{\substack{d ( \mbox{\tiny{\rm{mod\,}}}{c} ) \\ (c,d) = 1}}  v_\eta(M)^{-24m} e^{2\pi i \frac{ma+nd}{c}}, \; \;   
M = { \textstyle \big[{a \atop c}\,\,{* \atop d}\big] \in \Gamma  };$$     
$\sigma_{k-1}(n) = \sum_{ {d|n}, \, {d \,> \,0}} d^{k-1}$ is the divisor function; $\zeta$ symbolizes the Riemann zeta-function; and $B_k$ is the $k$th Bernoulli number.
(Note that $P_{k,0}$ is the normalized Eisenstein series of weight $k$, usually denoted by $E_k$, for which we have appended two familiar formulas
based upon core knowledge associated with the Ramanujan sum $A(0,n;c)$ and $\zeta(k)$.)  
\end{lem} 

We conclude this section by reproducing the following key fact (and its proof) from \cite[Section 3]{P3}.     
\begin{lem}
Let $F$ be an \textbf {unrestricted modular form} of weight $2$ and trivial multiplier on $\Gamma$. By this we mean that   
$F$ is holomorphic in $\mathbb{H}$ and satisfies the modular relation
$F(M\tau) = (c \tau + d)^2F(\tau)$, $\tau \in \mathbb{H}$,
for all $M = \big[{* \atop c}\,\,{* \atop d}\big] \in \Gamma$.      
So $F$ must have the expansion
$$ F(\tau) = \sum_{n \in \mathbb{Z}} a_n q^n,$$   
where $q = e^{2\pi i \tau}$ and  $\tau \in \mathbb{H}$. Then the constant term of $F$, given by $a_0$, is $0$. 
\end{lem}
\begin{proof} Rudimentary complex analysis reveals that
$$ a_0 = \int_{\rho}^{\rho + 1} F(\tau) \, d\tau  = \int_{\rho}^{\rho + 1} \frac{F(-1/\tau)}{\tau^2} \, d\tau =  \int_{\rho + 1}^{\rho} F(\zeta) \, d\zeta = -a_0,$$
where $\rho = e^{2\pi i/3}$ and the path is along the unit circle; so $a_0 = 0$.
\end{proof} Alternatively, Lemma 3 (for a modular form with a finite principal part) follows from the valence formula (plus some basic knowledge)
or from the residue theorem on the compact Riemann surface $\Gamma \backslash (\mathbb{H} \cup {\mathbb Q} \cup  \{ i\infty \})$
applied to the $\Gamma$-invariant differential $F \, d\tau$. The use of the latter 
can be interpreted as a special case of the Serre duality pairing (see \cite[Section 3]{B}). 

\section{Proof of Rademacher's formula}

\begin{thrm} Let $1 \le r \le 24$. Then the number of $r$-color partitions of $n \in \mathbb{N}$ is 
$$p_r(n) = -c_{r/24},$$
where $c_{r/24}$ is the coefficient of $q^{r/24}$ in the expansion for the Poincar\'e series $P_{2+r/2,-n+r/24}$ that is provided in Lemma 2.  
\end{thrm}
\begin{proof} Observe that the modular form 
$$P_{2+r/2,-n+r/24}(\tau) \cdot \eta(\tau)^{-r} = \Big(q^{-n+r/24} + c_{r/24} q^{r/24} + \cdots \Big) \Big( \sum_{\ell=0}^{\infty} p_r(\ell) q^{\ell - r/24} \Big)$$ 
is of weight $2$ and trivial multiplier on $\Gamma$, so its constant term $p_r(n) + c_{r/24}$ is $0$ by Lemma 3.
\end{proof}

The traditional form of Rademacher's formula as in Theorem 1 follows after 
applying the well-known identity $I_{3/2}(z) = \sqrt{\frac{2z}{\pi}} \frac{d}{dz}(\sinh(z)/z)$ and the relation $i^{-1/2} A(-n+1/24,1/24;c) = i^{1/2} A(-1/24,n-1/24;c)$. 
Actually, one can avoid altogether the use of $I_{3/2}$ by rewriting the second sum over $p$ specialized to $k = 5/2$---see the above computation of 
$\mathcal{I}$---in terms of $\sinh$. (For this easy and standard calculation, consult \cite[Equation~21]{P1}.) 
Amazingly, the formula for $p_{24}(n)$ follows already from the aforementioned work of Poincar\'e \cite{P} coupled with Lemma 3. 
Tangentially, we note the bilateral nature of the equality in Theorem 4. To wit, any known formula for $p_r(n)$ provides us with one for $c_{r/24}$ and, in combination
with Lemma 2, this can produce some curious identities. A simple illustration of this stems from the obvious fact that $p_r(1) = r$. Amusingly, the particular case
$p_{24}(1) = 24$ leads instantly to the value of $\zeta(14)$ (as well as that of $B_{14}$). 
Some additional familiar identities (connected with expansions of $0$) can be recovered by considering the product of $\eta^{-r}$, $1 \le r \le 24$,
with the cuspidal Poincar\'e series $P_{2+r/2,n+r/24}$, that is, where $n \in \mathbb{Z}_{\ge 0}$.  \\

We remark that the same method of this paper can be used to reestablish the Rademacher and Zuckerman expressions for the Fourier coefficients
of any modular form having negative real weight and multiplier system on $\Gamma$. 
(This includes a formula for $p_r(n)$, where $r,n \in \mathbb{N}$ with $n \ge r/24$, that depends upon the principal part of the Fourier expansion of $\eta^{-r}$.)
More broadly, an enhanced version of the method can be employed to capture explicit formulas for the Fourier coefficients of an arbitrary Niebur modular integral
(more recently also known as a mock modular form)
of negative real weight and multiplier system on $\Gamma$. This work has been carried out by the first author, who plans to present it in a forthcoming article.  \\

\textbf{Acknowledgments:} The first author is grateful for the influence of Donald J. Newman, problem solver {\it par excellence}, who taught him about partitions
some thirty years ago, and who instilled in him the notion that every mathematical truth has a simple proof. 
The second author thanks Ken Ono for his encouragement to write out this argument in detail. 

\bibliographystyle{plainnat}
\bibliography{\jobname}

\end{document}